\documentclass[11pt,reqno]{amsart}
\usepackage[T1]{fontenc}
\usepackage{lmodern}
\usepackage[a4paper,margin=27mm]{geometry}
\usepackage{amsmath,amssymb,amsthm,mathtools}
\usepackage{microtype,booktabs,array,longtable}
\usepackage{placeins}
\usepackage{xcolor}
\definecolor{linkblue}{RGB}{25,60,105}
\usepackage[colorlinks=true,linkcolor=linkblue,citecolor=linkblue,urlcolor=linkblue]{hyperref}
\usepackage[capitalise,nameinlink,noabbrev]{cleveref}
\newcommand{\N}{\mathbb N}
\newcommand{\Z}{\mathbb Z}
\newcommand{\Q}{\mathbb Q}
\newcommand{\F}{\mathbb F}
\newcommand{\C}{\mathbb C}
\newcommand{\T}{\mathbb T}
\newcommand{\E}{\mathbb E}
\newcommand{\one}{\mathbf 1}
\newcommand{\supp}{\operatorname{supp}}

\newcommand{\ord}{\operatorname{ord}}
\newtheorem{theorem}{Theorem}[section]
\newtheorem{proposition}[theorem]{Proposition}
\newtheorem{lemma}[theorem]{Lemma}
\newtheorem{corollary}[theorem]{Corollary}
\theoremstyle{definition}

\theoremstyle{remark}
\newtheorem{remark}[theorem]{Remark}
\numberwithin{equation}{section}
\title[Obstructions to coloring arithmetic graphs]{Obstructions to coloring arithmetic graphs}
\author{Lujia Wang and Ruihua Wang}
\address{Hainan Bielefeld University of Applied Sciences, Danzhou, China}
\email{\{lujia.wang,ruihua.wang\}@hibiuh.edu.cn}
\thanks{Lujia Wang is partially supported by the National Natural Science Foundation of China (grant No.~12371359).}
\date{14 September 2026}
\hypersetup{pdftitle={Obstructions to coloring arithmetic graphs},pdfauthor={Lujia Wang and Ruihua Wang}}
\begin{document}
\begin{abstract}
The arithmetic graph $B_n$ joins distinct $a,b\in\N$ when
$\max(a,b)/\gcd(a,b)\le n$. We prove $\chi(B_{205})=206$, disproving
the conjecture that $\chi(B_n)=n$ for every $n$, equivalently the
Rainbow Cascades Conjecture. The proof reduces an arbitrary tiling by
the arithmetic exponent tile to a periodic tiling, then to two families
of finite quotients, which are excluded using exact computations.
We also construct a $208$-coloring using $\Z_{104}\times\Z_2$ and prove
$212\le\chi(B_{211})\le213$. The lower bound at $211$ follows from
prime-cardinality tiling rigidity and the published nonexistence of a
cyclic logarithm of length $211$; we give a direct proof of the required
rigidity statement. Finally, we record the equivalence with the List
Cascade Coloring Conjecture and the conjecture on ironic decorations,
and deduce finite graph counterexamples to both. The least $n$ with
$\chi(B_n)>n$ is either $195$ or $205$; determining which remains open.
\end{abstract}
\maketitle
\section{Introduction}
Write $\N=\{1,2,\ldots\}$ and $[n]=\{1,\ldots,n\}$. A \emph{cascade}
of length $n$ is a set $r[n]=\{r,2r,\ldots,nr\}$, where $r,n\in\N$;
it is \emph{rainbow} if its elements receive pairwise distinct colors.
The arithmetic graph $B_n$ has vertex set $\N$ and adjacency relation
\[
 a\sim b\quad\Longleftrightarrow\quad
 a\ne b,\qquad \frac{\max(a,b)}{\gcd(a,b)}\le n.
\]
Every cascade $r[n]$ is a clique of order $n$, and every edge lies in the
cascade based at the greatest common divisor of its endpoints.
Thus $\chi(B_n)\ge n$, with equality precisely when there is a coloring
$c:\N\to[n]$ satisfying
\begin{equation}\label{eq:cascade-conjecture}
 \{c(r),c(2r),\ldots,c(nr)\}=[n]\qquad(r\in\N).
\end{equation}
Such colorings are called \emph{$n$-satisfactory} in~\cite{CCP}.
Bosek et al.~\cite{Bosek} conjectured that $\chi(B_n)=n$ for every $n\ge1$.
Grytczuk~\cite{Gryt} records the equivalent
\emph{Rainbow Cascades Conjecture}, which asserts that for each $n\ge1$
there is a coloring satisfying~\eqref{eq:cascade-conjecture}.
Pach posed the rainbow question around 2008--2009; it was posted on
MathOverflow by P\'alv\"olgyi in 2010; see~\cite{CCP,Gryt}.

Graham's original conjecture~\cite{Graham}, proved by Balasubramanian and Soundararajan~\cite{BS}, states that any $N$ distinct positive integers $a_1,\ldots,a_N$ satisfy
\begin{equation}\label{eq:graham}
 \max_{i,j}\frac{a_i}{\gcd(a_i,a_j)}\ge N.
\end{equation}
As observed in~\cite{Bosek}, this is equivalent to $\omega(B_n)=n$.
Our counterexamples show that the chromatic number of $B_n$ can exceed its clique number.

A familiar coloring construction assigns weights to prime exponents.
A \emph{cyclic logarithm of length $n$} is a bijection $h:[n]\to\Z_n$, where $\Z_n=\Z/n\Z$ is written additively, such that
\[
 h(ab)=h(a)+h(b)\qquad(ab\le n).
\]
The same definition with a finite abelian group $G$ of order $n$ gives a \emph{group logarithm}.
Forcade and Pollington~\cite{FP} identified $195$ as the first obstruction to the group construction.
Blackburn and McKee~\cite{BM} computed cyclic logarithms and recorded nonexistence at $195,205,208,211$.
Failure of the cyclic construction does not by itself exclude arbitrary colorings, as the positive example at $208$ makes explicit.
Mitchell developed further constructions of satisfactory colorings~\cite{Mitchell2019} and a permutation formulation linking rainbow cascades with hypercube tilings~\cite{Mitchell2023}.

The connection with translational tilings was already noted by Bosek et al.~\cite{Bosek} and developed by Caicedo, Chartier and Pach~\cite{CCP}; we recall it below.
Szegedy~\cite{Szegedy} established periodicity and the existence of
subgroup complements for generating tiles of prime cardinality;
Horak and Kim~\cite{HK} studied these questions by algebraic methods.
For arithmetic exponent tiles, these results make a prime-length
satisfactory coloring equivalent to a cyclic logarithm. The obstruction
at $211$ follows from this equivalence and the computation in~\cite{BM}.
The main new obstruction is at the composite value $205$, where a
reduction from arbitrary tilings to finite quotients is needed.

\begin{theorem}\label{thm:main}
The arithmetic graphs satisfy
\[
 \chi(B_{205})=206,\qquad
 \chi(B_{208})=208,\qquad
 212\le\chi(B_{211})\le213.
\]
If $n_0=\min\{n:\chi(B_n)>n\}$, then
\[
 n_0\in\{195,205\},\qquad 195\le\chi(B_{195})\le196.
\]
\end{theorem}

The paper is organized as follows. \Cref{sec:tiling} recalls the
connection between colorings, logarithms and exponent tilings.
\Cref{sec:prime} proves the prime-cardinality criterion and applies it
to the obstruction at $211$. \Cref{sec:205} proves the counterexample
at $205$, with the finite computational inputs and their verification
given in \cref{lem:205-computations} and the paragraphs following it.
\Cref{sec:208} gives the noncyclic coloring at $208$.
\Cref{sec:decorations} records the equivalence with the other coloring
conjectures using Kostochka's complete multipartite
construction~\cite{Kostochka}, compactness and a degree adjustment.
\Cref{sec:conclusion} discusses open questions.
Finally, \cref{app:positive} lists the weights for the upper bounds at
$205$ and $211$.

\section{Colorings, logarithms and exponent tilings}\label{sec:tiling}
Fix $n$, and let $r_1,\ldots,r_d$ be the primes at most $n$.
For a prime $r$, let $v_r(m)$ denote its exponent in the factorization
of $m$. Put
\[
 t(m)=(v_{r_1}(m),\ldots,v_{r_d}(m)),\qquad
 T_n=\{t(m):m\in[n]\}.
\]
The set $T_n$ has $n$ elements and contains $0$ and every coordinate unit vector.
An $n$-smooth number is an integer whose prime factors are all at most $n$; exponent vectors identify these integers with $\Z_{\ge0}^d$.

The following standard equivalence combines the smooth-number and tiling formulations of Caicedo, Chartier and Pach~\cite[Propositions~19, 25 and~26]{CCP}; see also~\cite{Bosek}.
\begin{samepage}
\begin{proposition}[Coloring--tiling correspondence]\label{prop:tiling}
The following are equivalent:
\begin{enumerate}
\renewcommand{\theenumi}{\roman{enumi}}
\renewcommand{\labelenumi}{\textup{(\theenumi)}}
\item $B_n$ has a proper $n$-coloring;
\item the graph on $\Z^d$ with difference set $(T_n-T_n)\setminus\{0\}$ has a proper $n$-coloring;
\item $T_n$ tiles $\Z^d$ by translations: there is $C\subset\Z^d$ such that every $x\in\Z^d$ has exactly one representation $x=t+c$ with $t\in T_n$, $c\in C$.
\end{enumerate}
\end{proposition}
\end{samepage}

We write $T_n\oplus C=\Z^d$ for the unique-representation property in
\cref{prop:tiling}, and call $C$ a \emph{tiling complement}.
A tiling complement is \emph{periodic} if $C+L=C$ for some subgroup
$L$ of finite index in $\Z^d$. A \emph{lattice tiling} is one whose
complement is a coset of a subgroup; by translation we may take the
complement to be the subgroup itself. Its index is $n$, with $T_n$ as
a complete set of coset representatives. A periodic complement may
instead be a union of several cosets of a period subgroup.

The group construction is also standard; see~\cite[Theorem~61 and Lemma~74]{CCP}.
\begin{lemma}[Group logarithm construction]\label{lem:group}
A group logarithm $h:[n]\to G$ gives a proper $n$-coloring of $B_n$.
Equivalently, a homomorphism $\phi:\Z^d\to G$ that is bijective on $T_n$ gives a tiling $T_n\oplus\ker\phi=\Z^d$.
\end{lemma}
The resulting coloring is given explicitly by
\[
 c(a)=\sum_{r\le n,\ r\text{ prime}}v_r(a)h(r)\qquad(a\in\N).
\]

\section{The prime case and the obstruction at 211}\label{sec:prime}
The following theorem is an arithmetic form of the prime-cardinality
tiling results of Szegedy~\cite{Szegedy} and Horak and Kim~\cite{HK}.
For completeness, we include a direct proof.
\begin{theorem}\label{thm:prime}
Let $q$ be prime, and let $\F_q$ be the field with $q$ elements.
The following statements are equivalent:
\begin{enumerate}
\renewcommand{\theenumi}{\roman{enumi}}
\renewcommand{\labelenumi}{\textup{(\theenumi)}}
\item $B_q$ admits a proper coloring with $q$ colors;
\item a logarithm of length $q$ exists;
\item there are coefficients $w_r\in\F_q$, indexed by primes $r\le q$, such that
\begin{equation}\label{eq:linear}
 h(m)=\sum_{\substack{r\le q\\r\text{ prime}}}w_r v_r(m)
\end{equation}
is a bijection from $[q]$ onto $\F_q$.
\end{enumerate}
Moreover, every coloring in \textup{(i)} satisfies
\begin{equation}\label{eq:periodcolor}
 c(am^q)=c(a)\qquad(a\in\N,\ 1\le m\le q).
\end{equation}
\end{theorem}

\begin{proof}
Suppose first that $c$ is a proper $q$-coloring. Fix a color and let
$f:\N\to\{0,1\}$ be its indicator. Since every cascade contains this
color once,
\begin{equation}\label{eq:cascade}
 \sum_{m=1}^{q}f(am)=1\qquad(a\in\N).
\end{equation}
For $m\in\N$, define the dilation operator $D_m$ by $(D_mf)(a)=f(am)$,
and put $S=D_1+\cdots+D_q$. These operators commute, with
$D_mD_k=D_{mk}$. Equation~\eqref{eq:cascade} says $Sf=1$; since
$S1=q$, we obtain $S^qf=q^{q-1}$.
In the multinomial expansion of $S^q$, every mixed coefficient is
divisible by the prime $q$. Reduction modulo $q$ therefore gives
\[
 S^q\equiv\sum_{m=1}^q D_m^q
      =\sum_{m=1}^q D_{m^q}\pmod q.
\]
Applying this congruence to $f$ yields
$\sum_{m=1}^q f(am^q)\equiv0\pmod q$. This sum consists of exactly
$q$ values in $\{0,1\}$, so it equals either $0$ or $q$.
All summands are therefore equal, and the term $m=1$ is $f(a)$.
Hence $f(am^q)=f(a)$ for every $a\in\N$ and $1\le m\le q$.
Applying the argument to each color proves~\eqref{eq:periodcolor}.

Take $n=q$ in the notation of \cref{sec:tiling}. On the $q$-smooth numbers,
the identity just proved makes each prime exponent periodic modulo $q$.
Consequently the color indicator descends to a function
$F:V\to\{0,1\}$ on $V=\F_q^d$, and~\eqref{eq:cascade} becomes
\begin{equation}\label{eq:torus}
 \sum_{m=1}^q F(x+t(m))=1\qquad(x\in V),
\end{equation}
where exponent vectors are reduced modulo $q$. The function $F$ is
nonconstant, since the cascade at $1$ contains the chosen color and
$q-1$ other colors.

Set $\zeta=e^{2\pi i/q}$, a primitive $q$th root of unity. For $\xi\in V$,
use the Fourier transform
\[
 \widehat F(\xi)=\sum_{x\in V}F(x)\zeta^{-\xi\cdot x}.
\]
Since $F$ is nonconstant, Fourier inversion supplies a nonzero $\xi$
with $\widehat F(\xi)\ne0$. Taking the Fourier coefficient of
\eqref{eq:torus} at this frequency, the constant right-hand side
contributes zero. Dividing by $\widehat F(\xi)$ gives
\begin{equation}\label{eq:rootsum}
 \sum_{m=1}^q\zeta^{\xi\cdot t(m)}=0.
\end{equation}
For $j\in\{0,\ldots,q-1\}$, let $N_j$ count the integers $m\in[q]$
for which $\xi\cdot t(m)=j$ in $\F_q$, and put
$A(X)=\sum_{j=0}^{q-1}N_jX^j$. Then $A(\zeta)=0$
by~\eqref{eq:rootsum}. Since the minimal polynomial of $\zeta$ over
$\Q$ is $\Phi_q(X)=1+X+\cdots+X^{q-1}$ and $\deg A\le q-1$,
the polynomial $A$ is a constant multiple of $\Phi_q$. Evaluating at
$1$ gives $A(1)=q=\Phi_q(1)$, so $A=\Phi_q$ and thus
$N_0=N_1=\cdots=N_{q-1}=1$. The map $h(m)=\xi\cdot t(m)$ is
consequently bijective on $[q]$. Valuation additivity gives
$h(ab)=h(a)+h(b)$ for $ab\le q$. This proves
\textup{(i)}$\Rightarrow$\textup{(iii)}$\Rightarrow$\textup{(ii)}.

Conversely, a logarithm has $h(1)=0$ and, by repeated factorization
within $[q]$, has the form~\eqref{eq:linear} with $w_r=h(r)$.
Given these coefficients, define on all positive integers
\[
 c(a)=\sum_{\substack{r\le q\\r\text{ prime}}}w_r v_r(a)\pmod q.
\]
If $a\sim b$, write $a=gx$ and $b=gy$, where $g=\gcd(a,b)$ and
distinct $x,y\in[q]$. Then $c(a)-c(b)=h(x)-h(y)\ne0$.
Thus $c$ is a proper $q$-coloring, proving the remaining implication.
\end{proof}

\begin{remark}
The theorem constructs a multiplicative coloring from the existence of an arbitrary coloring. It does not assert that the original coloring itself is multiplicative. Primality is used both in the operator congruence and in the cyclotomic argument.
\end{remark}

\begin{corollary}\label{cor:211}
The graph $B_{211}$ satisfies $212\le\chi(B_{211})\le213$.
\end{corollary}
\begin{proof}
Blackburn and McKee~\cite[\S9.3]{BM} report that no logarithm of length $211$ exists.
\cref{thm:prime} therefore excludes every $211$-coloring.
The prime weights in \cref{app:positive} define a logarithm of length $213$:
their values on $[213]$ are the $213$ distinct residues in $\Z_{213}$.
By \cref{lem:group} it colors $B_{213}$, hence its subgraph $B_{211}$, with $213$ colors.
\end{proof}

\section{The counterexample at \texorpdfstring{$205$}{205}}
\label{sec:205}

Throughout this section, $d=46$, and $e_r$ denotes the standard
basis vector indexed by the prime $r\le205$. Put
\[
 t(m)=(v_r(m))_{r\le205},\qquad T=\{t(1),\ldots,t(205)\}.
\]
The main task is to reduce an arbitrary tiling by $T$ to a finite
computation. We first state the finite arithmetic inputs in
\cref{lem:205-computations}. These give three periods of one factor
of the tiling equation, from which we construct a periodic tiling.
We then reduce its finite quotient to two families,
$\Z_5\times\Z_{41^b}$ and $\Z_{5^a}\times\Z_{41}$.
The first is excluded for every $b$; the second reduces to $a\le3$
and is excluded by the finite computation.

\begin{theorem}\label{thm:205}
The arithmetic graph satisfies $\chi(B_{205})=206$.
\end{theorem}

\subsection{Finite arithmetic inputs}\label{sec:finite-inputs}

The coordinates at $2,3,5,7,11,13$ identify $\Z^6$ with a subgroup of $\Z^d$.
There are $98$ integers in
\[
 I_0=\{m\in[205]:\text{every prime divisor of }m\text{ is at most }13\};
 \qquad T_0=\{t(m):m\in I_0\}\subset\Z^6.
\]
We call these the \emph{core labels} and the remaining $107$ labels the
\emph{tail}. Since $17^2>205$, each tail label has a unique expression
$rc$, with $r\ge17$ prime and $1\le c\le\lfloor205/r\rfloor\le12$.
Thus the tail consists of $40$ prime blocks whose cofactors all lie in $I_0$.

For prime weights $w_r$ in an additive abelian group $A$, write
\[
 W(m)=\sum_{\substack{r\le205\\ r\text{ prime}}}v_r(m)w_r,
 \qquad \mu_W(a)=\#\{m\in[205]:W(m)=a\}.
\]
We retain multiplicities in the image $W([205])$. When
$A=\Z\times\Z_{41}$, we call the first coordinate the \emph{height};
it is an integer, not a residue class.
For a core homomorphism $v:\Z^6\to\Z_{41}$, put
\[
 D_0(v)=\big\langle t(m)-t(n):m,n\in I_0,\quad
                 v(t(m))=v(t(n))\big\rangle\le\Z^6.
\]

We abbreviate $D_0(v)$ to $D_0$ when $v$ is fixed.
We use the following notation for Smith normal form.
For a subgroup $D\le\Z^k$ of rank $s$, Smith normal form gives a
basis $b_1,\ldots,b_k$ of $\Z^k$ and positive integers
$d_1\mid d_2\mid\cdots\mid d_s$ such that
\[
 D=\langle d_1b_1,\ldots,d_sb_s\rangle,
 \qquad \Z^k/D\cong\Z_{d_1}\oplus\cdots\oplus\Z_{d_s}\oplus\Z^{k-s}.
\]
The $d_i$ are the nonzero \emph{invariant factors} of $D$.

\begin{lemma}[Finite arithmetic checks; computer-assisted]\label{lem:205-computations}
The following statements hold.
\begin{enumerate}
\renewcommand{\theenumi}{\roman{enumi}}
\renewcommand{\labelenumi}{\textup{(\theenumi)}}
\item There is no homomorphism $(u,v):\Z^d\to\Z_{125}\times\Z_{41}$
such that, for every $b\in\Z_{41}$, the labels with $v(t(m))=b$ have
five distinct $u$-values forming one coset of $25\Z_{125}$.
Equivalently, each fiber contains exactly five labels, its $u$-values
agree modulo $25$, and no two agree modulo $125$.

\item For each $s\in\{0,1,2,3\}$, there is no choice of weights in
$\Z\times\Z_{41}$ for which $W([205])$ is supported on
$\{0,1,2,3\}\times\Z_{41}$, with
\[
 \mu_W(h,b)=1+\one_{\{h=s\}},
 \qquad (h,b)\in\{0,1,2,3\}\times\Z_{41},
\]
and weights
\[
 w_2=(0,1),\quad w_3=(0,b_3),\quad w_5=(1,0),\quad
 w_r\in\{0,1\}\times\Z_{41}\quad(r=7,11,13).
\]
For each $s\in\{4,5\}$, there is likewise no choice of weights in
$\Z_6\times\Z_{41}$ with
\[
 \mu_W(h,b)=\one_{\{h\ne s\}},\qquad
 w_2=(0,1),\quad w_3=(0,b_3),\quad w_5=(1,b_5).
\]
Here $b_3,b_5\in\Z_{41}$ are free, as are the other prime weights
unless explicitly restricted.

\item If $v:\Z^6\to\Z_{41}$ satisfies $v(e_2)=1$ and every
fiber of $v|_{T_0}$ has at most five labels, then $D_0(v)$ is
$5$-saturated in $\Z^6$:
\[
 5x\in D_0(v)\ \Longrightarrow\ x\in D_0(v)
 \qquad(x\in\Z^6).
\]
Equivalently, every nonzero invariant factor of $D_0(v)\le\Z^6$
is relatively prime to $5$.

\item Consider homomorphisms $L:\Z^6\to\Z_{205}$ whose
restrictions to $T_0$ are injective. Reducing $L$ modulo $41$ and
multiplying by a nonzero scalar to obtain $v(e_2)=1$ gives precisely
$51\,582$ distinct core maps $v$. For $51\,451$ of these,
$\operatorname{rank}D_0(v)=6$; the other $131$ have rank five.
For each rank-five map there is a primitive integer linear functional
$H:\Z^6\to\Z$ annihilating $D_0(v)$, with $H_r=H(e_r)$, such that
\[
 H_2\not\equiv0\pmod{41},\qquad H_r\equiv H_2v(e_r)\pmod{41},
\]
and
\[
 \max_{t\in T_0}H(t)-\min_{t\in T_0}H(t)\le58,
 \qquad
 \max_{1\le c\le12}H(t(c))-\min_{1\le c\le12}H(t(c))\le29.
\]
Here primitive means $\gcd(H_2,H_3,H_5,H_7,H_{11},H_{13})=1$.
\end{enumerate}
\end{lemma}

\begin{proof}[Computational verification of \cref{lem:205-computations}]
We describe the finite enumerations and give their outcomes. All tests
use exact integer or modular arithmetic.

For (i) and (ii), the variables are the prime weights. In (i), each
unrestricted weight ranges over $\Z_{125}\times\Z_{41}$.
In the four-height cases of (ii), each otherwise unrestricted weight
ranges over $\{0,1,2,3\}\times\Z_{41}$: its first coordinate must
be in this range because the prime itself is one of the labels.
In the six-height cases, the range is $\Z_6\times\Z_{41}$.
The prescribed weights and the restrictions at $7,11,13$ are imposed
as stated. For each normalization, the enumeration proceeds as follows.
\begin{enumerate}
\item Assign the weights at $2,3,5,7,11,13$ in that order. After
each assignment, evaluate every core label whose prime factors have
all been assigned. Reject the branch if a cell capacity is exceeded
or, in (i), if a fiber has incompatible residues modulo $25$ or a
repeated value modulo $125$.
In the four-height cases, also reject a branch as soon as an evaluated
label has height outside $\{0,1,2,3\}$.
\item For each surviving complete core, a prime $17\le r\le102$
contributes a translate of the already determined cofactor multiset
\[
 w_r+\bigl(W(1),\ldots,W(\lfloor205/r\rfloor)\bigr).
\]
Try every permitted shift, applying the same rejection conditions
to the combined labels. One may choose the block with the fewest
currently permitted shifts first; this changes only the search order.
\item After these blocks, the $20$ primes above $102$ are singleton
labels. Their unrestricted weights fill the remaining cell capacities;
in (i), they fill the unused top digits in each fiber, choosing any
common low residue if that fiber is still empty.
\end{enumerate}
Every full arithmetic assignment determines a branch in this search.
Each rejection is a violation that cannot be repaired by assigning
more prime weights. Conversely, a surviving completion has precisely
the required multiplicities, since the total capacity is $205$.
This proves that exhaustive termination decides the stated existence
questions. The enumerations in (iii) and (iv) stop at the core and
apply the matrix tests described below.

For (i), $v(e_2)=0$ would force $u(e_2)\equiv0\pmod{25}$ by the
labels $1,2$, and hence $u(t(32))=5u(e_2)=0=u(t(1))$, a collision.
Independent unit multiplications therefore normalize $v(e_2)=1$ and
$u(e_2)\in\{0,1,5,25\}$. Write a residue modulo $125$ as
$\ell+25j$, with $0\le\ell<25$ and $j\in\Z_5$.
The fiber test keeps a common low residue $\ell$ and distinct high
digits $j$. The numbers
of surviving complete core assignments are
\[
\begin{array}{c|rrrr}
 u(e_2)&0&1&5&25\\ \hline
 \text{complete cores}&60\,964&592&1\,012&47\,018\\
 \text{full completions}&0&0&0&0.
\end{array}
\]
All four tail searches terminate by exhaustion. An independently
implemented search reproduces both these core counts and the absence
of a completion.

In (ii), the doubled-height cases $s=0,1,2,3$ have respectively
$19,39,0,0$ complete cores and no full completion. The omitted-height
cases $s=4,5$ have $0,53$ complete cores. Each of the $53$ cores fails
every one of the $246$ possible shifts of the prime-$17$ block, already
excluding an image on its $98+12=110$ labels. The residue $b_5$ is
enumerated freely in this six-height computation.

For (iii), choose the least integer $m_b$ in each nonempty core
fiber $b$. Form the integer matrix $M(v)$ with rows
$t(m)-t(m_b)$ for $m\in I_0$ with $v(t(m))=b$.
These rows generate $D_0(v)$, since every same-fiber difference is
a difference of two such rows. The rank of $M(v)$ over $\F_5$
equals its rational rank if and only if
its nonzero invariant factors are all prime to $5$: in Smith coordinates
the two ranks count, respectively, the invariant factors not divisible
by $5$ and all nonzero invariant factors. This is also equivalent to
the stated saturation implication. Enumerate
$(v(e_3),v(e_5),v(e_7),v(e_{11}),v(e_{13}))\in\Z_{41}^5$,
rejecting a partial assignment once a core fiber contains more than
five labels. If its difference rows already have rank six modulo $5$,
every completion has rank six over both $\F_5$ and $\Q$, so that
branch satisfies (iii) without further enumeration.
At all other admissible leaves, compare the two ranks exactly.
This can be done by row reduction over $\Q$ and $\F_5$, or by
integer minor calculations. The ranks agree at every remaining leaf:
$1\,943\,876$ have rank five and $1\,283$ have rank four.
These are counts after the rank-six pruning step, not counts of
all completed core maps.

For (iv), injectivity on $1,2,4,\ldots,128$ forces the order of
$L(e_2)$ in $\Z_{205}$ to be $41$ or $205$. Unit multiplication
normalizes it to $5$ or $1$. Enumerating the remaining core weights
gives $60\,644$ and $46\,678$ injective core assignments, respectively.
Their normalized mod-$41$ projections, with repetitions removed, give
the $51\,582$ maps stated above. For each projection form $M(v)$ as
in (iii). Row reduction modulo $5$ gives rank six in $51\,451$ cases
and rank five in $131$ cases.

In a rank-five case, take five rows independent modulo $5$, forming
a $5\times6$ integer matrix $B$. Let $B^{(j)}$ be the matrix obtained
by deleting column $j$, and put
\[
 c_j=(-1)^{j+1}\det B^{(j)},\qquad
 H=\frac{(c_1,\ldots,c_6)}{\gcd(|c_1|,\ldots,|c_6|)}.
\]
The cofactor vector is nonzero, and its normalization is primitive.
In each of the $131$ cases, direct multiplication gives $M(v)H^{\mathsf T}=0$,
which also proves that the rational rank is five.
The coefficients satisfy the congruences in (iv). Evaluating the
resulting $H$ on all $98$ core labels and on $t(1),\ldots,t(12)$
gives the two stated width bounds. Thus the heights are constructed
from the enumerated core maps by exact integer calculations.
\end{proof}

The ancillary files contain the exact-arithmetic programs and completed
run logs for \cref{lem:205-computations}, together with a checker for
the weights in \cref{tab:208,tab:positive}. The file \texttt{anc/README.txt}
gives the commands and expected outputs.

The catalog in (iv) concerns projections of injective cyclic core maps.
It does not classify all capacity-five maps in (iii), which include
rank-four cases. The proof below uses saturation first to show that a
core arising in the finite-quotient argument belongs to this catalog.

An immediate consequence of (i), used below, is that no homomorphism
$\Z^d\to\Z_{205}$ is bijective on $T$: under
$\Z_{205}\cong\Z_5\times\Z_{41}$, multiplying the first coordinate
by $25$ would give an image excluded by (i).

\subsection{The three small periods}

Write $\T=\{z\in\C:|z|=1\}$ and
$\mu_s=\{z\in\T:z^s=1\}$. For $y\in\T$, a \emph{complete
$s$-gon} is the multiset $y\mu_s$, with each value counted once;
unions of such gons retain multiplicities. A \emph{character} of an
additive abelian group is a homomorphism into $\T$.
For a character $z:\Z^d\to\T$, write
\[
 z_r=z(e_r),\qquad \alpha(x)=z(x)^{41},\qquad
 \alpha_r=\alpha(e_r),\qquad E=\alpha(T).
\]
Here $E$ is a set, whereas $z(T)$ is counted with multiplicity.
We allow $\ord(u)=\infty$ for $u\in\T$ of infinite order.
The $41$st powers of the elements of a complete $41$-gon
$\rho\mu_{41}$ all equal $\rho^{41}$. Thus, if $z(T)$ is a union of
five such gons, then $|E|\le5$. If equality holds, each value in $E$
corresponds to exactly one gon, whose $41$ elements are distinct.

\begin{lemma}[Character values at small primes]\label{lem:205-phases}
Suppose the multiset $z(T)$ is a union of five complete $41$-gons.
Then
\[
 \operatorname{ord}(\alpha_2),\operatorname{ord}(\alpha_3),
 \operatorname{ord}(\alpha_5)\le4,
 \qquad z_2^{12}\in\mu_{41}\setminus\{1\}.
\]
\end{lemma}

\begin{proof}
We will repeatedly use the following consequence of
\cref{lem:205-computations}(i).
Suppose integer prime weights $k_r$ define the additive height
$K(x)=\sum_r k_rx_r$, whose values on $T$ lie in five consecutive
integers, and suppose these five heights distinguish the five values in $E$.
If a character $\beta:\Z^d\to\mu_{41}$ is obtained from $z$ by a
constant rotation on each height, then every height contains
each value of $\mu_{41}$ exactly once. Fix
$\zeta_{41}=e^{2\pi i/41}$ and write
$\beta(x)=\zeta_{41}^{b(x)}$, with $b:\Z^d\to\Z_{41}$ additive.
The homomorphism
\begin{equation}\label{eq:205-height-logarithm}
 x\longmapsto \bigl(K(x)\bmod5,b(x)\bigr)
       \in\Z_5\times\Z_{41}\cong\Z_{205}
\end{equation}
is then bijective on $T$: different heights have different first
coordinates, and within each height the second coordinates run through
$\Z_{41}$ once. Its restriction to the exponent labels is therefore
a logarithm of length $205$, contrary to
\cref{lem:205-computations}(i). The requirement that $K$ be an
\emph{additive integer} height matters: exponents chosen modulo the
order of a character value need not lift to such a height.

\smallskip\noindent
\emph{The values at $2$ and $3$.}
The eight labels $2^j$, $0\le j\le7$, show that $\alpha_2$ has finite order
at most five: an element of larger or infinite order would already
give six distinct values among its first six powers.
Suppose $u=\alpha_3$ has order $N>5$, including $N=\infty$.
The five labels $1,3,9,27,81$ then exhaust the set $E$:
\[
 E=\{1,u,u^2,u^3,u^4\}.
\]
Put $U=\langle\alpha_2\rangle$. Its order is at most five, and
the labels $2^j,3\cdot2^j,9\cdot2^j$, for $0\le j<|U|$, all
belong to $[205]$. Hence $U,uU,u^2U\subset E$.
This forces $U=\{1\}$. Indeed, $U\subset\langle u\rangle$ is
finite, so the assertion is immediate when $N=\infty$. For finite
$N\ge9$, a nontrivial subgroup of order $s$ contains
$u^{N-N/s}$, whose exponent is at least five and less than $N$.
For $N=6$, the possible nontrivial subgroups have order two or three:
$u^2U$ contains $u^5$ in the former case, and $uU$ contains $u^5$
in the latter. For $N=7$ there is no nontrivial subgroup of order at
most five. For $N=8$, the subgroup of order four already contains
$u^6$, while the subgroup of order two has $u^5\in uU$.
All possibilities contradict the displayed description of $E$.
Thus $\alpha_2=1$.

For each prime $r\le205$, choose the unique $k_r\in\{0,1,2,3,4\}$
such that $\alpha_r=u^{k_r}$. In particular $k_2=0$ and $k_3=1$.
For $r=5,7,11,13$, the labels $3r$ and $9r$ first force $k_r\le2$:
the choices three and four would produce the forbidden value $u^5$.
The labels $75=3\cdot5^2$ and $147=3\cdot7^2$ further exclude
$k_5=k_7=2$. Consequently
\[
 k_2=0,\quad k_3=1,\quad k_5,k_7\le1,\quad
 k_{11},k_{13}\le2.
\]
Define $K(x)=\sum_r k_rx_r$ using these integer weights.
For a core integer $m$, each core prime satisfies $r\ge3^{k_r}$,
so $m\ge3^{K(t(m))}$. Since $205<3^5$, this gives
$0\le K(t(m))\le4$ on the core.
For a tail $m=rc$, with $r\ge17$ prime and $c\le12$, the preceding
weights give $K(t(c))\le2$. If this cofactor height is one, then
$c\ge3$ and the available label $3r$ forces $k_r\le3$.
If it is two, then $c\ge9$ and the labels $3r,9r$ force $k_r\le2$.
These deductions use only exponents up to five, so no reduction modulo
$N$ occurs, even when $N=6$. Thus $K(t(m))\le4$ on every tail too.

All five heights occur at powers of three. Set
$\beta_r=z_r/z_3^{k_r}$ and extend multiplicatively to a character.
Since $\beta_r^{41}=\alpha_r/u^{k_r}=1$, this character takes values
in $\mu_{41}$. At height $j$ it divides every tile value by $z_3^j$,
so \eqref{eq:205-height-logarithm} gives the forbidden logarithm.
We conclude that $\operatorname{ord}(\alpha_3)\le5$.

If either $\alpha_2$ or $\alpha_3$ has order five, its observed powers
give $E=\mu_5$. Then $z(T)\subset\mu_{205}$, and its five gons are
the five distinct $\mu_{41}$-cosets of $\mu_{205}$, each used once.
The character $z$ itself therefore gives a logarithm to
$\mu_{205}\cong\Z_{205}$. This excludes order five at both primes.

The resulting bound $\operatorname{ord}(\alpha_2)\le4$ gives
$z_2^{12}\in\mu_{41}$. Suppose it were one. All eight values
$z(t(2^j))=z_2^j$ would then belong to $\mu_{12}$. Two such values in
the same $\mu_{41}$-coset must be equal, because
$\mu_{12}\cap\mu_{41}=\{1\}$. But one complete gon contains each
of its values only once. Hence each of the eight labels would need a
different gon, even if some values repeat. Five gons cannot suffice,
and $z_2^{12}\ne1$ follows.

\smallskip\noindent
\emph{The subgroup $\langle\alpha_2,\alpha_3\rangle$.}
Let $H=\langle\alpha_2,\alpha_3\rangle$. A finite subgroup of
$\T$ is cyclic, so its order is the least common multiple of the
two generator orders. Since both orders are at most four, the only
possibilities exceeding four are six and twelve. In those cases the
generator orders are respectively $2,3$ or $3,4$, in either order.
All their distinct products occur among $2^a3^b\le108$, contradicting
$|E|\le5$. Thus $|H|\le4$.

If $|H|=s\in\{3,4\}$, one of $\alpha_2,\alpha_3$ generates $H$.
Writing that prime as $r$, the labels $r^j,5r^j$ for $0\le j<s$
are at most $5\cdot3^3=135$ and show that
$H\cup\alpha_5H\subset E$. Distinct cosets would contribute
$2s>5$ values, so $\alpha_5\in H$.
If $|H|=2$, choose $r\in\{2,3\}$ with $\alpha_r\ne1$.
The six labels $r^i5^j$, $0\le i<2$, $0\le j<3$, are at most
$75$ and contain the three cosets $H,\alpha_5H,\alpha_5^2H$.
They cannot be distinct. The order of $\alpha_5H$ in
$\langle H,\alpha_5\rangle/H$ is therefore at most two, and the
whole group has order at most four. We have proved
\begin{equation}\label{eq:205-small-phase-subgroup}
 H\ne\{1\}\quad\Longrightarrow\quad
 |\langle\alpha_2,\alpha_3,\alpha_5\rangle|\le4.
\end{equation}
In particular, any remaining case with
$\operatorname{ord}(\alpha_5)>4$ has $\alpha_2=\alpha_3=1$.

\smallskip\noindent
\emph{The four initial values at $5$.}
Put $u=\alpha_5$ and assume its order is at least five or infinite.
The labels $1,5,25,125$ give the four distinct values
$I=\{1,u,u^2,u^3\}\subset E$. There are only two possibilities:
$E=I$, or $E=I\cup\{v\}$ for a single $v\notin I$.

First consider $E=I$. Write $\alpha_r=u^{k_r}$ with
$0\le k_r\le3$. For $r=7,11,13$, the labels $5r$ and $r^2$
exclude respectively $k_r=3$ and $k_r=2$, since $u^4\notin E$.
Thus the core weights are
\[
 k_2=k_3=0,\qquad k_5=1,\qquad
 k_7,k_{11},k_{13}\in\{0,1\}.
\]
For a core label, height at least four would require an integer at
least $5^4>205$. For a tail $rc$, the cofactor height is at most one,
so the unreduced height $k_r+K(t(c))$ is at most four. Its value
cannot be four because $u^4\notin E$. Thus the additive integer height
$K$ takes values only in $\{0,1,2,3\}$ on $T$. Each height corresponds
to at least one complete gon; as there are five gons in total, one height contains two
and the other three contain one each.

Set $\beta_r=z_r/z_5^{k_r}$ and write
$\beta_r=\zeta_{41}^{b_r}$. Here $b_5=0$. Also $\beta_2=z_2$ is a
nontrivial $41$st root: $\alpha_2=1$ and $z_2^{12}\ne1$.
Multiplying every $b_r$ by $b_2^{-1}$ normalizes $b_2=1$ and preserves
all cell multiplicities. The resulting arithmetic image in
$\{0,1,2,3\}\times\Z_{41}$ is exactly the doubled-height model
excluded by \cref{lem:205-computations}(ii). Thus $E\ne I$.

\smallskip\noindent
\emph{An additional value.}
Suppose now $E=I\cup\{v\}$. If
$v\notin\{u^4,u^{-1}\}$, the equations from $r,5r,r^2$ for
$r=7,11,13$ imply $\alpha_r\in\{1,u\}$.
Indeed, $u^2$ would require the missing square $u^4$, and $u^3$
would require the missing product $u^4$. A prime with value $v$ and
$5r\le205$ would require $uv\in E$; since $v\notin I$, this is
possible only when $uv=1$, or $v=u^{-1}$, which is excluded.
Consequently every prime with value $v$ exceeds $41$.

Give such exceptional primes height four. Give every other prime
the height $k_r\in\{0,1,2,3\}$ determined by $\alpha_r=u^{k_r}$.
A label containing an exceptional prime contains it once, has cofactor
at most four, and hence has height four and value $v$.
For all other labels, the same core and tail estimates as above bound
the unreduced height by four; height four is excluded by $u^4\notin E$.
Thus these labels have heights $0,1,2,3$. The five heights correspond
bijectively to the five values in $E$.

Choose $\theta\in\T$ with $\theta^{41}=v$ and put
\[
 \beta_r=
 \begin{cases}
 z_r/\theta,&\alpha_r=v,\\
 z_r/z_5^{k_r},&\alpha_r\ne v.
 \end{cases}
\]
These prime values lie in $\mu_{41}$ and define one character
$\beta$. On heights $0,1,2,3$ its denominator is $z_5^K$; on height
four its denominator is $\theta$, because the exceptional prime occurs
once and its cofactor has height zero. Thus each height still contains
a complete $41$-gon after division. The logarithm
\eqref{eq:205-height-logarithm} is again a contradiction.
This argument applies to every order at least five, including infinite
order, whenever the additional value is neither $u^4$ nor $u^{-1}$.

It remains to consider $v=u^4$ and $v=u^{-1}$. First suppose
$\operatorname{ord}(u)\ge7$, or $u$ has infinite order.
If $v=u^4$, use prime heights in $[0,4]$.
The labels $5r,r^2$ show that the core primes $r=7,11,13$ have
heights at most two: height three would require both $u^4$ and $u^6$,
and height four would require $u^5$. A core label of height at least
five would therefore have size at least
\[
 \min\{5^5,5^3\cdot7,5\cdot7^2,7^3\}=245>205.
\]
A tail has cofactor height at most two and hence unreduced height at
most six. These seven powers are distinct, so membership in $E$ forces
the height to be at most four.
If $v=u^{-1}$, use prime heights in $[-1,3]$. A core prime cannot
have height $-1$, since its square would have the absent value $u^{-2}$;
heights two and three would require the absent value $u^4$ at $r^2$
or $5r$. Thus $k_7,k_{11},k_{13}\in\{0,1\}$.
Core heights are in $[0,3]$, and a tail has unreduced height in
$[-1,4]$. These powers are distinct; since $u^4\notin E$, its height
lies in $[-1,3]$. In both cases the five consecutive heights
are attained and correspond to the five values in $E$. Dividing by $z_5^K$
produces a character in $\mu_{41}$ and the forbidden logarithm
\eqref{eq:205-height-logarithm}. Together with the preceding cases,
this proves $\operatorname{ord}(\alpha_5)\le6$.

For order five, an additional value in $\mu_5$ makes $E=\mu_5$,
which directly gives the logarithm already excluded above. An additional
value outside $\mu_5$ is covered by the exceptional-prime construction,
and $E=I$ was excluded by the doubled-height computation.
For order six, the same arguments exclude $E=I$ and an additional
value outside $\mu_6$. The two remaining possibilities are
\[
 E=\mu_6\setminus\{u^4\}
 \quad\hbox{or}\quad E=\mu_6\setminus\{u^5\}.
\]
Use a coordinate in $\Z_6$. Since $\gcd(41,6)=1$, choose
$\theta\in\mu_6$ with
$\theta^{41}=u$. There is an additive map
$k:\Z^d\to\Z_6$ such that $\alpha(x)=u^{k(x)}$, and
$\beta(x)=z(x)/\theta^{k(x)}$ is a well-defined character into
$\mu_{41}$. Writing $\beta=\zeta_{41}^b$ gives an arithmetic image
in $\Z_6\times\Z_{41}$. Each of the five permitted heights contains
one complete gon, so every cell outside the one omitted height is
filled exactly once. As before, normalize $b_2=1$; the prime weights
have the form
\[
 w_2=(0,1),\qquad w_3=(0,b_3),\qquad w_5=(1,b_5).
\]
These are precisely the two six-height models excluded by
\cref{lem:205-computations}(ii). Here $b_5$ is unrestricted, since
$\theta$ must lie in $\mu_6$ whereas $z_5$ need not.
Orders five and six are excluded, completing the proof.
\end{proof}

\subsection{From partial periods to a periodic tiling}\label{sec:periodic-replacement}

Let $A\subset\Z^d$ contain zero, with $|A|=pq$ for distinct primes
$p,q$, and suppose $A\oplus C=\Z^d$. The reflected indicator
$f_0=\one_{-C}$ satisfies
\begin{equation}\label{eq:cotiler}
 \sum_{a\in A}f_0(x+a)=1\qquad(x\in\Z^d).
\end{equation}
Conversely, a $\{0,1\}$-valued solution determines a complement
$C=-\supp f_0$. As for sets, a function is called periodic if it is
invariant under a subgroup of finite index in $\Z^d$.
We will pass to an invariant probability space and obtain a
factorization $f=gh$ of its coordinate function. Suitable partial
periods of $h$ will then yield a periodic tiling.

Let $X$ be the orbit closure of $f_0$ in $\{0,1\}^{\Z^d}$, with
the product topology and shifts
$(\omega+a)(x)=\omega(x+a)$. Each $\omega\in X$ satisfies the tiling
equation, since the equation at any lattice point involves finitely
many coordinates. Averaging the point masses at translates of $f_0$
over expanding integer boxes and taking a weakly convergent subsequence
gives an invariant probability measure $\mathbb P$ on $X$.
Set $f(\omega)=\omega(0)$ and
$U_aF(\omega)=F(\omega+a)$. Translation invariance and the tiling
equation give $\E f=1/(pq)$.

For an integer $k$, put
\[
 S_kF(\omega)=\sum_{a\in A}F(\omega+ka),\qquad
 A_k(z)=\sum_{a\in A}z(a)^k,
\]
where $z:\Z^d\to\T$ is a character.
The dilation theorem of Horak and Kim~\cite[Theorem~8]{HK}, or its
measurable version~\cite[Theorem~1.2(i)]{GGRT}, gives
\begin{equation}\label{eq:205-coprime-dilations}
 S_kf=1\qquad\text{if }(k,pq)=1.
\end{equation}
The spectral theorem for the commuting unitary operators $U_a$
associates to each $F\in L^2(X,\mathbb P)$ a positive measure
$\sigma_F$ on the character group $\T^d$, satisfying
\[
 \|P(U)F\|_2^2=\int |P(z)|^2\,d\sigma_F(z),
\]
where $P(U)$ is any finite linear combination of translations and
$P(z)$ is the corresponding combination of character values. Thus an
equation $P(U)F=0$ means that $P(z)=0$ for $\sigma_F$-almost every
character. Applied to $\widetilde f=f-1/(pq)$,
\eqref{eq:205-coprime-dilations} shows that
$A_k(z)=0$ for every $k$ coprime to $pq$, simultaneously for almost
every character in its spectral measure.

For such a character, the multiset $\{z(a):a\in A\}$ consists either
of $q$ complete $p$-gons or of $p$ complete $q$-gons.
To prove this, let
$\nu=\sum_{a\in A}\delta_{z(a)}$, and let $R_s$ average a measure over
its rotations by $\mu_s$. Its $k$th Fourier coefficient is retained
when $s\mid k$ and is zero otherwise. The vanishing power sums,
together with their complex conjugates, therefore imply
\[
 (I-R_p)(I-R_q)\nu=0.
\]
Indeed, all Fourier coefficients of this signed measure vanish;
uniqueness follows from the density of trigonometric polynomials in
the continuous functions on $\T$.

On each coset of $\mu_{pq}$ meeting the support of $\nu$, index the
multiplicities by $\Z_p\times\Z_q$. The displayed identity says that
the resulting table $M_{ij}$ has zero mixed differences:
$M_{ij}-M_{ij_0}-M_{i_0j}+M_{i_0j_0}=0$.
Choose $i_0$ minimizing the column $M_{ij_0}$. Then
\[
 M_{ij}=a_i+b_j,\qquad
 a_i=M_{ij_0}-M_{i_0j_0}\ge0,\quad b_j=M_{i_0j}\ge0,
\]
with all entries integral. The table is consequently a sum of full
rows and full columns, hence of complete $p$-gons and $q$-gons.
If their total numbers are $P,Q$, then $pP+qQ=pq$.
Reduction modulo $p$ and $q$ gives only
$(P,Q)=(q,0)$ or $(0,p)$, proving the assertion. This argument uses
all coprime power sums; it does not require the individual values
$z(a)$ to be roots of unity.

This alternative implies $A_p(z)A_q(z)=0$: a complete $p$-gon has
zero $q$th power sum, and conversely. Returning to the translation
operators gives $S_pS_q\widetilde f=0$, and therefore
$S_pS_qf=pq$. There is also a pointwise divisibility statement.
Frobenius for the commuting shifts, reduced modulo $p$, gives
\[
 S_pf\equiv S_1^{\,p}f=(pq)^{p-1}\equiv0\pmod p;
\]
similarly $q\mid S_qf$.

Fix $\omega$ and form the $\{0,1\}$-matrix
$M_{ab}=f(\omega+pa+qb)$, with rows indexed by $a\in A$ and columns by
$b\in A$. Its total number of ones is $pq$. Every nonzero row sum is
at least $q$, and every nonzero column sum is at least $p$. Hence
there are at most $p$ nonempty rows and at most $q$ nonempty columns.
Conversely, any nonempty column requires at least $p$ rows, and any
nonempty row requires at least $q$ columns. Thus there are exactly
$p$ such rows and $q$ such columns. Their $pq$ intersections account
for all $pq$ ones, so every intersection is occupied. It follows that
\begin{equation}\label{eq:205-rectangle}
 \begin{gathered}
 g=p^{-1}S_pf\in\{0,1\},\qquad
 h=q^{-1}S_qf\in\{0,1\},\qquad f=gh,\\
 f(\omega+pa+qb)=g(\omega+qb)h(\omega+pa)\qquad(a,b\in A).
 \end{gathered}
\end{equation}
Here the column sum at $b$ is $p\,g(\omega+qb)$, the row sum at $a$ is
$q\,h(\omega+pa)$, and $a=b=0$ gives $f=gh$.
Also $\E g=1/p$ and $\E h=1/q$.

The spectral measure of $g-1/p=p^{-1}S_p\widetilde f$ is obtained
by multiplying that of $\widetilde f$ by $|A_p|^2/p^2$.
This removes the characters of the $q$-gon type. Thus the remaining
characters for $g-1/p$ have the $p$-gon form; those for $h-1/q$
have the $q$-gon form. In particular,
\begin{equation}\label{eq:205-factor-dilations}
 S_kh=p\quad(q\nmid k),\qquad S_kg=q\quad(p\nmid k).
\end{equation}
All these identities hold almost everywhere. Since there are only
countably many identities and lattice translates, we may restrict to
one invariant set of full measure on which they hold simultaneously.

\begin{samepage}
\begin{lemma}[Coprime period extension]\label{lem:205-orbit}
Let $A$ generate $\Z^d$, and let $f,g,h$ be the functions on the
invariant probability space just constructed, including the spectral
properties above. Suppose that $h$ is unchanged almost everywhere by all translations
in $N\Lambda$, where $\Lambda\le\Z^d$, $N\ge1$, and $(N,p)=1$.
If one coset of $\Lambda$ contains more than $q$ points of $A$, then
every element of $q\Z^d$ is a period of $h$ almost everywhere, and
$A$ admits a periodic tiling complement.
\end{lemma}
\end{samepage}

\begin{proof}
Put $K=N\Lambda$. A character trivial on $K$ sends each coset of
$\Lambda$ into a single coset of $\mu_N$. Since $(N,p)=1$, a complete
$p$-gon meets such a coset in at most one point. A union of $q$
$p$-gons therefore cannot account for the specified more than $q$
labels, even with multiplicities. Consequently the spectral measure
of $g-1/p$ has no character trivial on $K$.

Let $\E_K$ denote conditional expectation onto the functions unchanged
by every $K$-translation; in $L^2$ this is the orthogonal projection
onto those functions. The preceding spectral observation and the
$K$-invariance of $h$ imply
\[
 \E_Kg=1/p,\qquad \E_Kf=\E_K(hg)=h/p.
\]
It follows that, for almost every $\omega$ with $h(\omega)=1$, some $K$-translate
$\omega'$ of $\omega$ satisfies $f(\omega')=1$. To justify this without assumptions on
individual orbit sizes, the event
$\{\omega:f(\omega+k)=0\text{ for every }k\in K\}$ is $K$-invariant;
on that event $\E_Kf=0$, so $h=0$.

For $t\in A$, apply~\eqref{eq:205-rectangle} at $\omega'-qt$ with
$a=0$ and $b=t$. Since $f(\omega')=1$, this gives $h(\omega'-qt)=1$.
The $K$-periods of $h$ then give $h(\omega-qt)=1$. Hence
$h(\omega)\le h(\omega-qt)$ almost everywhere. Translation preserves the
integral, so equality holds almost everywhere. As $A$ generates
$\Z^d$, this proves every period in $q\Z^d$.

Choose a configuration $\omega_0$ on which these periods and all
preceding identities hold at every lattice translate. Restrict $f,g,h$
to its orbit, writing $f(x)=f(\omega_0+x)$ and similarly for $g,h$.
We now construct a periodic replacement for this function on $\Z^d$.
Set
\[
 K=q\Z^d,\qquad k=p+q.
\]
Then $(k,pq)=1$, $k\equiv p\pmod q$, $S_kf=1$, and $S_kh=p$.
Make an undirected graph on $\Z^d$ by joining $y$ to $y-pa$ whenever
$h(y)=1$ and $a\in A$. The rectangle identity gives
$g(y)=g(y-pa)$ along such an edge, so $g$ is constant on each
connected component. Since $h$ is $K$-periodic, translation by $K$
preserves this graph. Its quotient modulo $K$ is finite.

Fix a connected component $C$ lying above one quotient component,
choose $c\in C$, and set
\[
 D=\{z\in K:C+z=C\},\qquad B=K/D.
\]
Every component above this quotient component is a translate $C+z$,
and the translates are indexed by $B$. Their values of $g$ define
a function $b:B\to\{0,1\}$, with
$b(z+D)=g(c+z)$.

For each vertex of the quotient graph, choose a representative $x$
and $b_x\in K$ such that $x$ is connected to $c+b_x$.
Exactly $p$ labels $a\in A$ satisfy $h(x+ka)=1$, because $S_kh=p$.
For any $z\in K$ and each selected $a$, the graph edge from
$x+z+ka$ ends at $x+z+qa$. Consequently
\[
 g(x+z+ka)=b(z+b_x+qa+D).
\]
Thus $S_kf(x+z)=1$ becomes
\begin{equation}\label{eq:205-component-tiles}
 \sum_{\substack{a\in A\\h(x+ka)=1}}
 b\bigl(\bar z+(b_x+qa+D)\bigr)=1\qquad(\bar z\in B).
\end{equation}
The $p$ indicated shifts are distinct: a repeated shift would make
the sum at least two after translating it to a point where $b=1$.
There is such a point by the equation itself. We have therefore
obtained finitely many $p$-point sets in the finitely generated
abelian group $B$, whose tiling equations have the common solution $b$.

We claim that these equations have a common periodic $\{0,1\}$-valued
solution. The following argument applies to any finite family of
$p$-point sets in a finitely generated abelian group, including groups
with torsion. Translate each set separately to contain zero, and let
$L$ be the subgroup generated by all these normalized sets. For one
such set $V$, put $R=\sum_{v\in V}U_v$. Since $Rb=1$,
Frobenius gives
\[
 \sum_{v\in V} b(x+pv)\equiv R^p b(x)=p^{p-1}
 \equiv0\pmod p.
\]
The sum consists of $p$ values in $\{0,1\}$, one of which is $b(x)$ because
$0\in V$. They are therefore all equal, proving every period
in $p\langle V\rangle$, including when the ambient group has torsion.
Doing this for each set gives all periods in $pL$.

The restriction of $b$ to $L$ consequently descends to the finite
group $L/pL$. This subgroup of $B/pL$ embeds in a finite quotient
of $B/pL$: a finitely generated abelian group is residually finite,
so finitely many finite quotients separate the nonzero elements of
$L/pL$, and their product suffices. Copy the solution on $L/pL$
onto each coset of its image in that finite quotient and pull back
to $B$. This gives the claimed periodic solution.

Replace $b$ by a common periodic solution for each of the finitely
many quotient components. Assign these new values to the corresponding
components $C+z$, obtaining $g'$, and put $f'=hg'$. All equations
\eqref{eq:205-component-tiles} are preserved, hence $S_kf'=1$.
The period subgroup for each new function on $K/D$ lifts to a subgroup
of finite index in $K$. Intersecting these finitely many subgroups
shows that $f'$ is periodic. Finally,
\[
 F(x)=f'(kx)\qquad(x\in\Z^d)
\]
is periodic and satisfies
$\sum_{a\in A}F(x+a)=S_kf'(kx)=1$. Thus $-\supp F$ is a periodic
tiling complement for $A$.
\end{proof}

Apply this with $A=T\subset\Z^{46}$, $p=5$, $q=41$. The spectral
measure of $h-1/41$ has the five-complete-$41$-gon form, so
\cref{lem:205-phases} gives $z_r^{492}=1$ for $r=2,3,5$:
each order at most four divides $12$, and $492=41\cdot12$.
The spectral norm identity then gives the periods
\[
 492e_2,\qquad492e_3,\qquad492e_5
\]
of $h$ almost everywhere. There are $25+14+6+1=46$ labels of $T$
in $\Lambda=\Z e_2+\Z e_3+\Z e_5$, counted according to their
power of five. Since $(492,5)=1$ and $46>41$,
\cref{lem:205-orbit} therefore shows that if $T$ tiles $\Z^d$, then
it admits a periodic tiling complement.

\subsection{Finite quotients and two prime-power families}\label{sec:205-quotients}

By the preceding subsection, a tiling of $\Z^d$ by $T$ has a periodic
replacement. Passing to the quotient by a period subgroup of finite
index in $\Z^d$ gives a tiling of a finite abelian group, with all
$205$ tile labels distinct.
Indeed, if two tile points had the same image, a translate of the
tiling equation would contain two summands equal to one.
We next reduce this finite tiling to two cyclic prime-power coordinates.
The argument works in finite abelian groups; the related criteria
of Coven and Meyerowitz~\cite{CM} concern tiles in $\Z$.

\begin{lemma}[Two-prime quotient]\label{lem:205-pq-quotient}
If a $pq$-point set $A$ tiles a finite abelian group, where $p<q$ are
primes, it has an injective homomorphic image in
$\Z_{p^a}\times\Z_{q^b}$, for some $a,b\ge1$, with tiling complement
\begin{equation}\label{eq:205-box}
 \{0,\ldots,p^{a-1}-1\}\times\{0,\ldots,q^{b-1}-1\}.
\end{equation}
If $A$ is a homomorphic image of an arithmetic exponent tile containing
zero and every prime basis vector, one may further arrange $a=1$ or $b=1$.
\end{lemma}

\begin{proof}
We first find the two cyclic coordinates by Fourier analysis, and then
use the geometry of the complement to reduce one exponent to one.

\emph{Finding a rectangular complement.}
Coprime dilation preserves a finite tiling and injectivity of its labels,
by the dilation identity~\cite[Theorem~8]{HK} applied to its periodic
lift to an integer lattice. Choose a dilation that kills all primary
components other than those at $p$ and $q$.
The resulting image tiles the product $P\times Q$, where $P$ is a
finite abelian $p$-group and $Q$ is a finite abelian $q$-group.
Write $A\oplus B=P\times Q$, retaining $A$ for its injective image.

Write $\widehat A$ for the Fourier transform of $\one_A$, so that
$\widehat A(\chi)=\sum_{a\in A}\overline{\chi(a)}$.
Index characters by $(u,v)\in\widehat P\times\widehat Q$, with $0$
denoting a trivial character. If $\widehat A(u,v)=0$, the same holds
for every power of $(u,v)$ coprime to $pq$, by Galois conjugacy.
The polygon decomposition from the preceding subsection therefore gives
\begin{equation}\label{eq:205-pure-zero}
 \widehat A(u,v)=0\quad\Longrightarrow\quad
 \widehat A(u,0)=0\ \text{or}\ \widehat A(0,v)=0.
\end{equation}
Indeed, the character values are either $q$ complete $p$-gons or
$p$ complete $q$-gons. Projection to the corresponding primary roots
gives $\widehat A(u,0)=0$ in the first case and
$\widehat A(0,v)=0$ in the second.

Put $Z_p=\{u:\widehat A(u,0)=0\}$ and
$Z_q=\{v:\widehat A(0,v)=0\}$.
On functions on $P\times Q$, let $E_p$ be the orthogonal Fourier
projection retaining the coefficients with $u\in Z_p$, and let $E_q$
retain those with $v\in Z_q$. Put $\Pi_p=I-E_p$ and $\Pi_q=I-E_q$.
The convolution kernels of $E_p,\Pi_p$ are rational numbers with
denominators powers of $p$; the analogous denominators for $E_q,\Pi_q$
are powers of $q$. To check this, the kernel numerator is a sum of
roots of unity invariant under all Galois automorphisms, because the
sets of zeros of $\widehat A$ are Galois invariant. Such a sum
is both rational and an algebraic integer, hence an integer.

Suppose $\widehat A(u,v)\ne0$ for every $(u,v)\in Z_p\times Z_q$.
The tiling identity implies that the nonconstant Fourier coefficients
of $f=\one_B$ are supported on the zeros of $\widehat A$.
Consequently \eqref{eq:205-pure-zero} and this supposition give
\[
 E_pE_qf=0,\qquad \Pi_p\Pi_qf=\frac1{pq},\qquad
 \Pi_pf+\Pi_qf=f+\frac1{pq}.
\]
The last identity separates a rational number with a $p$-power
denominator from one with a $q$-power denominator. Modulo integers
it forces $\Pi_pf\in a/p+\Z$ pointwise, where $1\le a<p$ and
$aq\equiv1\pmod p$. Every such value has absolute value at least $1/p$.
Using normalized counting measure and the fact that an orthogonal
projection does not increase the $L^2$ norm, we obtain the contradiction
\[
 \frac1{p^2}\le\|\Pi_pf\|_2^2\le\|f\|_2^2
 =\frac1{pq}<\frac1{p^2}.
\]

There are therefore characters $u,v$ of orders $p^a,q^b$ such that
$(u,0),(0,v),(u,v)$ are all zeros. They define a homomorphism to
$\Z_{p^a}\times\Z_{q^b}$ by writing their values as powers of primitive
roots of the corresponding orders.
The interval-box indicator in \eqref{eq:205-box} has Fourier support
\[
 (\{0\}\cup\Z_{p^a}^{\times})\times
 (\{0\}\cup\Z_{q^b}^{\times}),
\]
where $\Z_s^{\times}$ denotes the invertible residues modulo $s$.
The three zero conditions and their independent Galois conjugates
make the Fourier transform of the image multiplicity vanish at every
nonzero frequency in this support.
Thus convolution of the image multiplicity with the box indicator is
the constant function one. All multiplicities are nonnegative integers,
so the image is injective and the box is its tiling complement.

\emph{Reducing one coordinate.}
For the final assertion, assume that $A$ is the image of an arithmetic
exponent tile. In particular, the image contains zero and the images
of all the prime basis vectors.
Put $L=p^{a-1}$ and $M=q^{b-1}$, and call the image points the
translation centers of the boxes. Any two centers differ by a nonzero
multiple of $L$ in the first coordinate or a nonzero multiple of $M$
in the second. For if neither condition held, multiplication by suitable
units in the two coordinates would move both differences into the
interval difference sets and make the boxes overlap. These unit
multiplications preserve the Fourier zero conditions and hence the same
box complement, a contradiction.
Here the unit choice is elementary: a residue modulo $p^a$ that is
not a nonzero multiple of $L$ is either zero or a unit multiple of
$p^r$ with $r<a-1$. It can therefore be moved into
$\{0,\ldots,L-1\}$; the same reasoning applies to the other coordinate.

Partition the group into the $pq$ cells
\[
 \{iL,\ldots,(i+1)L-1\}\times
 \{jM,\ldots,(j+1)M-1\},
 \quad 0\le i<p,\quad 0\le j<q.
\]
The separation condition permits at most one center in each cell;
there are $pq$ centers, so every cell contains one. Write that center
as $(iL+\varphi_{ij},jM+\psi_{ij})$, with
$0\le\varphi_{ij}<L$ and $0\le\psi_{ij}<M$.
Two cells in the same row force their first remainders to agree, and
two in the same column force their second remainders to agree.
Thus the centers have the form
$(iL+\varphi(j),jM+\psi(i))$.
For distinct rows and columns, separation says
$\varphi(j)=\varphi(j')$ or $\psi(i)=\psi(i')$.
If $\varphi$ is nonconstant, comparison of two rows where it differs
forces $\psi$ to be constant. Hence at least one of these two functions
is constant, and that constant is zero because the image contains zero.

If the first remainder is always zero, every prime basis vector maps
to a first coordinate divisible by $L$. Division by $L$ therefore
defines a homomorphism into $\Z_p\times\Z_{q^b}$ and reduces the box
complement to $\{0\}\times\{0,\ldots,M-1\}$. If the second remainder
is zero, the same argument reduces the second coordinate to $\Z_q$.
This proves the final assertion.
\end{proof}

For $205=5\cdot41$, \cref{lem:205-pq-quotient} leaves two families
of injective images with complement~\eqref{eq:205-box}:
$\Z_5\times\Z_{41^b}$ and $\Z_{5^a}\times\Z_{41}$.
The exponents $a,b$ are initially unbounded. The next two lemmas
exclude these families.

\begin{lemma}[$41$-power family]\label{lem:205-41-family}
For every $b\ge1$, there is no homomorphism
$(v,w):\Z^d\to\Z_5\times\Z_{41^b}$ that is injective on $T$ and
whose image of $T$ has tiling complement
$\{0\}\times\{0,\ldots,41^{b-1}-1\}$.
\end{lemma}

\begin{proof}
Suppose such an image $(v,w)$ exists, and define the character
$z(t)=\exp(2\pi i w(t)/41^b)$.
For a fixed first coordinate $j\in\Z_5$, its interval translates
partition the circle $\Z_{41^b}$ into $41$ intervals of length
$41^{b-1}$. Their starts are therefore equally spaced by $41^{b-1}$,
so $z$ takes one complete $41$-gon on that fiber.
By \cref{lem:205-phases}, each of $z_2^{41},z_3^{41},z_5^{41}$ has
order at most four. Its order also divides a power of $41$, so all
three values are one. Taking $41$st powers is constant on each
fiber; hence there is a function $F:\Z_5\to\T$ with
\[
 z(t)^{41}=F(v(t)),\qquad F(0)=1.
\]
If $v(e_2)$ or $v(e_3)$ is nonzero, the labels $1,r,r^2,r^3,r^4$
for that prime $r$ cover all five residues and have $z^{41}=1$.
Thus $F=1$ and $(v,z)$ is a forbidden logarithm into
$\Z_5\times\mu_{41}\cong\Z_{205}$.
Otherwise $v(e_2)=v(e_3)=0$ and $v(e_5)\ne0$: if the latter also
vanished, all $46$ labels supported on $2,3,5$ would lie in one
fiber of size $41$. Normalize $v(e_5)=1$.
The labels $1,5,25,125$ now give $F(0)=F(1)=F(2)=F(3)=1$.
If also $F(4)=1$, we again have the forbidden logarithm.

Suppose instead that $F(4)=\eta\ne1$. Every prime of $v$-value four
must exceed $41$: otherwise multiplication by five would put a label
with $z^{41}$-value $\eta$ in the zeroth fiber, whose $z^{41}$-value is one.
Such a prime occurs to the first power and with a cofactor at most four.
Let $K$ be the integer linear functional counting these exceptional
prime factors. On $T$, it is exactly the indicator of the fourth
fiber. Indeed, a label with no such prime has $z^{41}$-value one and cannot
belong to that fiber, whereas a label containing one has $v$-value four.
Choose $\theta$ with $\theta^{41}=\eta$. The character
$z\theta^{-K}$ takes values in $\mu_{41}$ and is obtained from $z$
by a constant rotation on each fiber. It preserves the complete
$41$-gon there and, with $v$, gives a forbidden logarithm.
This excludes the family for every $b\ge1$.
\end{proof}

\begin{lemma}[$5$-power family]\label{lem:205-5-family}
For every $a\ge1$, there is no homomorphism
$(u,v):\Z^d\to\Z_{5^a}\times\Z_{41}$ that is injective on $T$ and
whose image of $T$ has tiling complement
$\{0,\ldots,5^{a-1}-1\}\times\{0\}$.
\end{lemma}

\begin{proof}
Suppose such an image $(u,v)$ exists.
Write $T_j=\{t\in T:v(t)=j\}$. Each $T_j$ consists of five points,
and their $u$-values have a common residue modulo $5^{a-1}$ and
five distinct \emph{top digits}: after subtracting that common residue,
they are $5^{a-1}j$ for $j\in\Z_5$. The eight powers of two therefore force
$v(e_2)\ne0$; normalize $v(e_2)=1$.
Let $D\le\Z^d$ be generated by all differences $t-t'$ with
$v(t)=v(t')$, and let $D_0\le\Z^6$ be generated by the corresponding
core differences.

On $D_0$, the values of $u$ lie in the subgroup
$5^{a-1}\Z_{5^a}\simeq\Z_5$.
Write $\bar u:D_0\to\Z_5$ for their top digits.
By \cref{lem:205-computations}(iii), every nonzero invariant factor
of $D_0$ is prime to $5$. Expressing $D_0$ in Smith normal form and
dividing by these factors in $\Z_5$ extends $\bar u$ to $\Z^6$.
Together with $v$, this extension distinguishes all core labels.
Thus its cyclic image belongs to the catalog in
\cref{lem:205-computations}(iv). In particular, the rank-four maps
allowed in (iii) cannot arise here.

Put $Q=\Z^d/D$ and
$Q_{(5)}=Q\otimes_{\Z}\Z_{(5)}$, where
$\Z_{(5)}=\{a/b\in\Q:5\nmid b\}$.
Passing to $Q_{(5)}$ makes multiplication by every integer prime to
$5$ invertible. Since the invariant factors of $D_0$ are prime to $5$,
the localized core quotient is zero in the rank-six case and is free
of rank one in the rank-five case. Thus an element $\tau\in Q_{(5)}$ satisfies
\[
 [t(m)]=H(t(m))\tau\qquad(m\in I_0).
\]
Here $H$ is the primitive height in (iv), extended by zero on the
remaining prime coordinates. In the rank-six case take $H=0$ and
$\tau=0$. The element $\tau$ generates the image of the localized core, but
need not have infinite order after the tail relations are imposed.

To control the tail primes, form an abelian group $P$ with generators
$F_j$ for $j\in\Z_{41}$ and one generator $\tau$, subject to $F_0=0$
and the relations
\[
 F_{v(t(m))}=H(t(m))\tau\quad(m\in I_0),
\]
\[
 F_{v(e_r)+v(t(c))}-F_{v(e_r)}=H(t(c))\tau
 \quad\left(r\ge17\text{ prime},\ 1\le c\le\lfloor205/r\rfloor\right).
\]
The unique core--tail decomposition gives a homomorphism
$Q\to P$ sending $[t(m)]$ to $F_{v(t(m))}$.
Indeed, the core relations preserve addition of exponent vectors,
and every other label is $t(rc)=e_r+t(c)$ with $c\le12$.
Conversely, evaluating $F_j$ at the common image in $Q_{(5)}$ of the
labels in $T_j$, and evaluating $\tau$ at the element above, gives
$P\to Q_{(5)}$. The composite $Q\to P\to Q_{(5)}$ is localization.

View each relation $F_j-F_i=h\tau$ as an oriented edge of a
multigraph on the $41$ residues; retain loops and parallel edges.
Every edge satisfies
\[
 |h|\le58,\qquad h\equiv H_2(j-i)\pmod{41}.
\]
Choose a spanning forest. Its relations express each $F_j$ as a
component offset plus an integer multiple of $\tau$.
Every remaining edge therefore gives a relation $k\tau=0$.
Its forest path has at most $40$ edges, so its fundamental cycle has
at most $41$, including the cases of a loop or parallel edges. Hence
\[
 41\mid k,\qquad |k|\le41\cdot58.
\]
If $k\ne0$, writing $k=41\ell$ gives $0<|\ell|\le58<125$, and thus
$v_5(k)\le2$. Eliminating the forest variables identifies $P$ with a direct sum
of a free abelian group and $\Z/J$, where $J$ is the ideal generated
by the integers $k$ from the remaining edges. The subgroup of elements
of $\Z/J$ killed by a power of $5$ is therefore killed by $25$.
If every $k$ is zero, the cyclic group is also free.

Let $x\in Q$ be killed by a power of $5$. Its image in $P$ is killed
by the same power, so $25x$ maps to zero in $Q_{(5)}$.
Some integer prime to $5$ therefore kills $25x$ in $Q$.
B\'ezout's identity with the power of $5$ killing $x$ gives $25x=0$.
Equivalently,
\[
 5^n x\in D\quad\Longrightarrow\quad25x\in D
 \qquad(x\in\Z^d, n\ge0).
\]
Thus the nonzero invariant factors of $D$ have five-adic valuation at most two.

Choose a basis $b_1,\ldots,b_d$ putting $D$ in Smith normal form:
$D=\langle d_1b_1,\ldots,d_sb_s\rangle$ and
$d_i=5^{t_i}q_i$, where $t_i\le2$ and $5\nmid q_i$.
Suppose first that $a>3$. Since $5u(d_i b_i)=0$,
there is $c_i\in\Z_{5^{t_i+1}}$ such that
$u(b_i)=5^{a-1-t_i}c_i$. Define
$u'(b_i)=5^{2-t_i}c_i\in\Z_{125}$ for $i\le s$,
and put $u'(b_i)=0$ on the unused basis vectors.
For $a\le3$, instead put $u'=5^{3-a}u$.
Both constructions preserve the top-digit map on $D$:
\[
 u'(d)=25r\quad\text{whenever}\quad
 u(d)=5^{a-1}r,\qquad d\in D,\ r\in\Z_5.
\]
All within-fiber differences lie in $D$. Hence the five $u'$-values
in each fiber remain distinct and have a common residue modulo $25$.
This contradicts \cref{lem:205-computations}(i), excluding the family
for every $a\ge1$.
\end{proof}

\begin{proof}[Proof of \cref{thm:205}]
A $205$-coloring gives a tiling by $T$ by \cref{prop:tiling}.
The three small periods and \cref{lem:205-orbit} give a periodic
tiling complement, hence an injective tiling image in a finite abelian group.
\Cref{lem:205-pq-quotient} gives an injective image with one of the
interval complements excluded by \cref{lem:205-41-family,lem:205-5-family}.
Therefore
$\chi(B_{205})\ge206$. The explicit logarithm of length $206$ in
\cref{app:positive}, combined with \cref{lem:group}, gives the reverse inequality.
\end{proof}

\section{A noncyclic coloring at 208}\label{sec:208}

Absence of a cyclic logarithm does not by itself obstruct an optimal
coloring. The following explicit example also corrects the inclusion of
$208$ in the list of groupless numbers in~\cite[Table~5.1]{CCP}.

\begin{proposition}\label{prop:208}
The arithmetic graph $B_{208}$ has chromatic number $208$.
\end{proposition}

\begin{proof}
Let $G=\Z_{104}\times\Z_2$.
For every prime $p\le208$, let $w_p\in G$ be the weight in
\cref{tab:208}, and define
\[
 c(a)=\sum_{\substack{p\le208\\p\text{ prime}}}v_p(a)w_p
 \qquad(a\ge1).
\]
Evaluating the displayed weights on the factorizations of
$1,\ldots,208$ gives every element of $G$ exactly once. Since
$c(am)=c(a)+c(m)$, every cascade $\{a,2a,\ldots,208a\}$ is rainbow.
Every edge lies in a cascade, so $c$ is a proper $208$-coloring.
The clique $[208]$ gives the reverse inequality.
\end{proof}

\begin{table}[htbp]
\centering
\caption{Weights in $\Z_{104}\times\Z_2$
for the $208$-coloring.}\label{tab:208}
\renewcommand{\arraystretch}{1.08}
\setlength{\tabcolsep}{8pt}
\begin{tabular}{rr@{\hspace{12pt}}rr@{\hspace{12pt}}rr}
\toprule
$p$ & $w_p$ & $p$ & $w_p$ & $p$ & $w_p$\\
\midrule
2   & $(1,0)$  & 59  & $(65,1)$  & 137 & $(96,0)$\\
3   & $(81,1)$ & 61  & $(14,1)$  & 139 & $(57,1)$\\
5   & $(45,1)$ & 67  & $(56,0)$  & 149 & $(84,0)$\\
7   & $(4,1)$  & 71  & $(40,0)$  & 151 & $(32,1)$\\
11  & $(75,0)$ & 73  & $(27,1)$  & 157 & $(73,0)$\\
13  & $(14,0)$ & 79  & $(42,1)$  & 163 & $(34,1)$\\
17  & $(98,0)$ & 83  & $(55,1)$  & 167 & $(99,1)$\\
19  & $(90,1)$ & 89  & $(70,0)$  & 173 & $(9,1)$\\
23  & $(70,1)$ & 97  & $(37,0)$  & 179 & $(74,1)$\\
29  & $(19,0)$ & 101 & $(102,0)$ & 181 & $(61,1)$\\
31  & $(43,0)$ & 103 & $(11,1)$  & 191 & $(63,0)$\\
37  & $(53,0)$ & 107 & $(13,1)$  & 193 & $(25,1)$\\
41  & $(81,0)$ & 109 & $(78,1)$  & 197 & $(51,1)$\\
43  & $(33,0)$ & 113 & $(2,1)$   & 199 & $(38,1)$\\
47  & $(64,0)$ & 127 & $(29,1)$  &     & \\
53  & $(26,0)$ & 131 & $(69,1)$  &     & \\
\bottomrule
\end{tabular}
\end{table}

\begin{proof}[Completion of the proof of \cref{thm:main}]
The results proved above establish the assertions at $205,208,211$.
Blackburn and McKee~\cite{BM} record cyclic logarithms for every
$n\le204$ except $195$. By \cref{lem:group}, $\chi(B_n)=n$ at all
these parameters, so the least counterexample is either $195$ or $205$.
Since $197$ is prime, the prime-minus-one construction
\cite[Theorem~1]{CCP} gives $\chi(B_{196})=196$.
The inclusion $B_{195}\subseteq B_{196}$ and the clique $[195]$
therefore give $195\le\chi(B_{195})\le196$.
\end{proof}

\FloatBarrier
\begin{samepage}
\section{Equivalent coloring conjectures}\label{sec:decorations}

We record the equivalence of arithmetic graph coloring, rainbow
cascades, cascade-list colorings and ironic decorations, and apply it
to the counterexamples at $205$ and $211$.

An \emph{ironic decoration} of a graph $G$ is a labeling $f$ for which
$d_G(u)f(u)\ne d_G(v)f(v)$ on every edge, where $d_G(v)$ is the degree
of $v$. A \emph{cascade-list coloring} is a proper vertex coloring whose
value at $v$ belongs to its prescribed list $r_v[n]$.

The four assertions below are the fixed-$n$ forms of Conjectures~6--9
in Grytczuk's survey~\cite{Gryt}. The implications from rainbow
cascades to cascade-list colorings and then to ironic decorations,
and the equivalence with arithmetic graph coloring, are given
there~\cite[\S2.6, \S3.1 and \S3.3]{Gryt}.
Bosek et al.~\cite[Theorem~6]{Bosek} also derive ironic decorations
from arithmetic graph colorings, calling them \emph{fictional colorings}.
The proof below includes the converse reductions and identifies the
constructions from the literature that they use.

\begin{proposition}[Equivalent coloring assertions]\label{prop:cascade-equivalence}
For each fixed $n\ge2$, the following statements are equivalent.
\begin{enumerate}
\renewcommand{\theenumi}{\roman{enumi}}
\renewcommand{\labelenumi}{\textup{(\theenumi)}}
\item Every finite graph $G$ with $\chi(G)=n$ has an ironic decoration
$f:V(G)\to[n]$.
\item Every finite graph $G$ with $\chi(G)=n$ has a proper coloring
from any prescribed lists $L(v)=r_v[n]$, where $r_v\in\N$.
\item There is a coloring $c:\N\to[n]$ making every cascade $r[n]$ rainbow.
\item $\chi(B_n)=n$.
\end{enumerate}
All four statements also hold for $n=1$.
\end{proposition}
\end{samepage}

\begin{proof}
The equivalence $\textup{(iii)}\Longleftrightarrow\textup{(iv)}$
is the clique observation preceding~\eqref{eq:cascade-conjecture}.
For $\textup{(iii)}\Rightarrow\textup{(ii)}$, fix a proper
$n$-coloring of $G$ and choose from each list the unique integer
whose color under the rainbow coloring agrees with the color of its
vertex. This is a proper list coloring. For
$\textup{(ii)}\Rightarrow\textup{(i)}$,
isolated vertices can be labeled arbitrarily; the lists at nonisolated
vertices are $d_G(v)[n]$. We give the two converse reductions explicitly.

\emph{From cascade lists to rainbow cascades: \textup{(ii)}$\Rightarrow$\textup{(iii)}.}
The finite step is the complete multipartite construction in
Kostochka~\cite[proof of Theorem~2]{Kostochka}, specialized to cascades.
His \emph{panchromatic} condition requires every set to contain every
color; for $n$-element sets and $n$ colors, this is precisely the
rainbow condition.
Let $\mathcal F$ be a nonempty finite family of length-$n$ cascades.
Form the complete $n$-partite graph with parts
\[
 V_i=\{v_{i,L}:L\in\mathcal F\},\qquad i\in[n],
\]
so vertices in different parts are adjacent and vertices in the same
part are not. Each part is nonempty, hence this graph has chromatic
number $n$. Give $v_{i,L}$ the list $L$, and let $f$ be a proper list
coloring supplied by \textup{(ii)}. Define
\[
 P_i=\{f(v):v\in V_i\}.
\]
The sets $P_i$ are pairwise disjoint, since vertices in different
parts are adjacent. For each $L\in\mathcal F$, its $n$ copies receive
$n$ distinct members of $L$, and therefore exhaust $L$. It follows that
\[
 |L\cap P_i|=1\qquad(L\in\mathcal F,\ i\in[n]).
\]
Assigning color $i$ to the integers in $P_i$ gives a rainbow coloring
of every member of $\mathcal F$.

To pass from finite families to all cascades, we use compactness.
For each $M$, apply the preceding construction to the finite family of cascades
contained in $[M]$, coloring any remaining integers arbitrarily;
an empty family imposes no constraint.
The valid colorings of these initial intervals form a tree under
restriction, with at most $n$ extensions of each coloring and a node
at every depth. K\"onig's infinity lemma gives an infinite compatible
sequence of colorings. Every cascade is contained in some initial
interval, so the resulting coloring of $\N$ proves \textup{(iii)}.

\emph{From ironic decorations to cascade lists: \textup{(i)}$\Rightarrow$\textup{(ii)}.}
We prescribe degrees by attaching leaves, a device also used by
Dehghan, Sadeghi and Ahadi~\cite{DSA} in their study of fictional coloring.
Let $G$ be a finite graph with $\chi(G)=n$, with lists $r_v[n]$.
Choose an integer $M$ larger than every vertex degree in $G$.
For each vertex $v$, attach $Mr_v-d_G(v)$ new vertices, each adjacent
only to $v$, and call the resulting graph $H$. Then
\[
 d_H(v)=Mr_v\qquad(v\in V(G)),\qquad \chi(H)=n.
\]
Indeed, $G$ is a subgraph of $H$, and any $n$-coloring of $G$ extends
to the new degree-one vertices because $n\ge2$.
By \textup{(i)}, choose an ironic decoration $f:V(H)\to[n]$.
For every original edge $uv$, cancellation of $M$ in
\[
 Mr_uf(u)=d_H(u)f(u)\ne d_H(v)f(v)=Mr_vf(v)
\]
gives $r_uf(u)\ne r_vf(v)$. Thus $v\mapsto r_vf(v)$ is a proper
coloring of $G$ from its assigned lists, proving \textup{(ii)}.
\end{proof}

\begin{corollary}\label{cor:decorations}
For each $n\in\{205,211\}$, all four assertions in
\cref{prop:cascade-equivalence} fail. In particular, there are finite
graphs $G,H$ with $\chi(G)=\chi(H)=n$ such that $G$ has an uncolorable
assignment of cascade lists and $H$ has no ironic decoration with
labels in $[n]$. Thus Conjectures~6--9 of~\cite{Gryt} are false.
\end{corollary}

\begin{proof}
By \cref{thm:main}, $\chi(B_n)>n$ for $n=205,211$.
Apply \cref{prop:cascade-equivalence}. Its proof allows $G$ to be
chosen complete $n$-partite, and $H$ to be obtained from $G$ by
attaching degree-one vertices.
\end{proof}

\section{Conclusion and open questions}\label{sec:conclusion}

The graphs $B_{205}$ and $B_{211}$ have chromatic number greater
than the clique number given by Graham's theorem and its Farey-sequence
formulation~\cite{BS,Wang}. Although rainbow cascade colorings need
not exist, a weaker balance property holds:
for every fixed $n$, there is a two-coloring of $\N$ in which the two
color counts differ by at most one on every length-$n$
cascade~\cite[Theorem~1]{BG}.

The immediate questions are whether $\chi(B_{195})=195$ or $196$,
and whether $\chi(B_{211})=212$ or $213$.
Forcade--Pollington's obstruction excludes group logarithms at
$195$~\cite{FP}, but does not settle arbitrary colorings. The missing
step at $195$ is a reduction of an arbitrary tiling by $T_{195}$ to a
bounded finite search: the argument at $205=5\cdot41$ uses the
two-prime polygon decomposition to obtain a periodic replacement in
\cref{sec:periodic-replacement}, and then bounds the relevant
prime-power quotients in \cref{sec:205-quotients}. Neither reduction
has been established here for $195=3\cdot5\cdot13$; excluding group
logarithms, or searching quotients up to a prescribed size, therefore
does not exclude a $195$-coloring.

More generally, when does an arithmetic exponent tile that tiles by
translations admit a periodic complement, or a subgroup complement?
Prime cardinality gives affirmative answers. The character equations
connect these questions with vanishing sums of roots of unity~\cite{LL}.
One-dimensional tiling criteria~\cite{CM} do not settle
higher-dimensional periodic replacement, and the periodic tiling
conjecture fails in general~\cite{GT}.

\section*{Contributions and use of AI}
The human authors identified $211$ as a counterexample. They used
GPT-6 Astra (OpenAI) to assist with the argument at $205$, the construction
at $208$, computations, and manuscript preparation. The proof at $205$
is computer-assisted, with its finite arithmetic inputs stated in
\cref{lem:205-computations}; the proof at $211$ uses the published
logarithm obstruction of Blackburn and McKee~\cite{BM}.
The authors independently repeated the exhaustive search at $211$,
confirming that no cyclic logarithm of this length exists.
The authors take full responsibility for the proofs, computations,
and content of this paper.

\clearpage
\appendix
\section{Explicit logarithms for the upper bounds}\label{app:positive}
For $n=206,213$, let $w_{n,p}$ be the weights in \cref{tab:positive} and set
\[
 h_n(m)=\sum_{p\le n,\ p\text{ prime}}w_{n,p}v_p(m)\pmod n.
\]
Direct evaluation gives every residue in $\Z_n$ exactly once, and
valuation additivity gives $h_n(ab)=h_n(a)+h_n(b)$ for $ab\le n$.
Thus these weights define logarithms of lengths $206$ and $213$.
By \cref{lem:group}, they give the upper bounds at $205$ and $211$.
The ancillary checker verifies bijectivity from these weights.
\begin{table}[htbp]
\centering
\caption{Cyclic logarithm weights for the upper bounds. A dash means that the prime exceeds the length.}\label{tab:positive}
\renewcommand{\arraystretch}{1.05}
\setlength{\tabcolsep}{6pt}
\begin{tabular}{rrr@{\hspace{18pt}}rrr@{\hspace{18pt}}rrr}
\toprule
$p$&$w_{206,p}$&$w_{213,p}$&$p$&$w_{206,p}$&$w_{213,p}$&$p$&$w_{206,p}$&$w_{213,p}$\\
\midrule
2&1&1&59&144&39&137&102&91\\
3&10&122&61&159&53&139&108&96\\
5&68&140&67&193&87&149&112&100\\
7&121&81&71&25&192&151&126&111\\
11&93&11&73&55&64&157&166&132\\
13&42&55&79&60&78&163&170&147\\
17&47&118&83&82&98&167&176&150\\
19&172&180&89&106&102&173&181&157\\
23&177&16&97&129&105&179&184&160\\
29&75&36&101&152&165&181&195&170\\
31&117&167&103&197&184&191&200&174\\
37&99&72&107&27&20&193&201&197\\
41&155&114&109&29&30&197&202&198\\
43&64&24&113&33&44&199&205&208\\
47&148&186&127&63&61&211&--&210\\
53&18&200&131&91&75&&&\\
\bottomrule
\end{tabular}
\end{table}

\FloatBarrier


\clearpage
\begin{thebibliography}{99}

\bibitem{Bosek}
B. Bosek, M. D\k{e}bski, J. Grytczuk, J. Sok\'o\l{},
M. \'Sleszy\'nska-Nowak, and W. \.Zelazny,
\emph{Graph coloring and Graham's greatest common divisor problem},
Discrete Mathematics \textbf{341} (2018), no.~3, 781--785.
\href{https://doi.org/10.1016/j.disc.2017.11.006}{doi:10.1016/j.disc.2017.11.006}.

\bibitem{Gryt}
J. Grytczuk,
\emph{From the 1-2-3 conjecture to the Riemann hypothesis},
European Journal of Combinatorics \textbf{91} (2021), 103213.
\href{https://doi.org/10.1016/j.ejc.2020.103213}{doi:10.1016/j.ejc.2020.103213}.

\bibitem{BM}
S.~R. Blackburn and J.~F. McKee,
\emph{Constructing $k$-radius sequences},
Mathematics of Computation \textbf{81} (2012), no.~280, 2439--2459.
\href{https://doi.org/10.1090/S0025-5718-2011-02510-X}{doi:10.1090/S0025-5718-2011-02510-X}.
\href{https://arxiv.org/abs/1006.5812}{arXiv:1006.5812}.

\bibitem{CCP}
A.~E. Caicedo, T.~A.~C. Chartier, and P.~P. Pach,
\emph{Coloring the $n$-smooth numbers with $n$ colors},
Electronic Journal of Combinatorics \textbf{28} (2021), no.~1, Paper~P1.34.
\href{https://doi.org/10.37236/8492}{doi:10.37236/8492}.

\bibitem{Mitchell2019}
L. Mitchell,
\emph{Generating infinitely many satisfactory colorings of positive integers},
Discrete Mathematics \textbf{342} (2019), no.~12, 111602.
\href{https://doi.org/10.1016/j.disc.2019.111602}{doi:10.1016/j.disc.2019.111602}.

\bibitem{Mitchell2023}
L. Mitchell,
\emph{Rainbow Cascades and permutation-labeled hypercube tilings},
Examples and Counterexamples \textbf{3} (2023), 100099.
\href{https://doi.org/10.1016/j.exco.2023.100099}{doi:10.1016/j.exco.2023.100099}.

\bibitem{HK}
P. Horak and D. Kim,
\emph{Algebraic Method in Tilings},
\href{https://arxiv.org/abs/1603.00051v1}{arXiv:1603.00051v1} (2016).

\bibitem{Szegedy}
M. Szegedy,
\emph{Algorithms to tile the infinite grid with finite clusters},
Proceedings of the 39th Annual Symposium on Foundations of Computer Science,
IEEE, 1998, 137--147.
\href{https://doi.org/10.1109/SFCS.1998.743437}{doi:10.1109/SFCS.1998.743437}.

\bibitem{GGRT}
J. Greb\'ik, R. Greenfeld, V. Rozho\v{n}, and T. Tao,
\emph{Measurable tilings by abelian group actions},
International Mathematics Research Notices \textbf{2023}, no.~23, 20211--20251.
\href{https://doi.org/10.1093/imrn/rnad048}{doi:10.1093/imrn/rnad048}.

\bibitem{GT}
R. Greenfeld and T. Tao,
\emph{A counterexample to the periodic tiling conjecture},
Annals of Mathematics \textbf{200} (2024), no.~1, 301--363.
\href{https://doi.org/10.4007/annals.2024.200.1.5}{doi:10.4007/annals.2024.200.1.5}.

\bibitem{BG}
B. Bosek and J. Grytczuk,
\emph{Reflections on the Erd\H{o}s Discrepancy Problem},
\href{https://arxiv.org/abs/2005.14283}{arXiv:2005.14283} (2020).

\bibitem{Graham}
R.~L. Graham,
\emph{Advanced Problem 5749},
American Mathematical Monthly \textbf{77} (1970), no.~7, 775.

\bibitem{BS}
R. Balasubramanian and K. Soundararajan,
\emph{On a conjecture of R.~L. Graham},
Acta Arithmetica \textbf{75} (1996), no.~1, 1--38.
\href{https://doi.org/10.4064/aa-75-1-1-38}{doi:10.4064/aa-75-1-1-38}.

\bibitem{FP}
R.~W. Forcade and A.~D. Pollington,
\emph{What is special about 195? Groups, $n$th power maps and a problem of Graham},
in \emph{Number Theory (Banff, AB, 1988)},
R.~A. Mollin (ed.), de Gruyter, Berlin, 1990, 147--155.
\href{https://doi.org/10.1515/9783110848632-015}{doi:10.1515/9783110848632-015}.

\bibitem{Wang}
L. Wang,
\emph{Farey sequence and Graham's conjectures},
Journal of Number Theory \textbf{229} (2021), 399--404.
\href{https://doi.org/10.1016/j.jnt.2020.10.013}{doi:10.1016/j.jnt.2020.10.013}.
\href{https://arxiv.org/abs/2005.04429}{arXiv:2005.04429}.

\bibitem{CM}
E.~M. Coven and A.~D. Meyerowitz,
\emph{Tiling the integers with translates of one finite set},
Journal of Algebra \textbf{212} (1999), no.~1, 161--174.
\href{https://doi.org/10.1006/jabr.1998.7628}{doi:10.1006/jabr.1998.7628}.

\bibitem{LL}
T.~Y. Lam and K.~H. Leung,
\emph{On vanishing sums of roots of unity},
Journal of Algebra \textbf{224} (2000), no.~1, 91--109.
\href{https://doi.org/10.1006/jabr.1999.8089}{doi:10.1006/jabr.1999.8089}.

\bibitem{Kostochka}
A. Kostochka,
\emph{On a theorem of Erd\H{o}s, Rubin, and Taylor on choosability of
complete bipartite graphs},
Electronic Journal of Combinatorics \textbf{9} (2002), no.~1, Note~N9, 4~pp.
\href{https://doi.org/10.37236/1670}{doi:10.37236/1670}.

\bibitem{DSA}
A. Dehghan, M.-R. Sadeghi, and A. Ahadi,
\emph{Algorithmic complexity of proper labeling problems},
Theoretical Computer Science \textbf{495} (2013), 25--36.
\href{https://doi.org/10.1016/j.tcs.2013.05.027}{doi:10.1016/j.tcs.2013.05.027}.
\href{https://arxiv.org/abs/1701.06669}{arXiv:1701.06669}.

\end{thebibliography}
\end{document}